\documentclass[11pt]{amsart}
\usepackage[english]{babel}
\usepackage{amsfonts, amsmath, amsthm, amssymb,amscd,indentfirst}

\usepackage{mathtools}
\usepackage{extarrows}

\usepackage{tikz}

\numberwithin{equation}{section}

\usepackage{marginnote}

\usepackage{mathrsfs}
\usepackage{euscript}

\usepackage{xcolor}
\usepackage{amsmath,amssymb,latexsym,indentfirst}
\usepackage{times}
\usepackage{palatino}

\usepackage{stmaryrd}

\usepackage{wasysym}

\usepackage{graphicx}
\graphicspath{ {./images/}}

\theoremstyle{plain}
\newtheorem{theorem}{Theorem}[section]

\newtheorem{proposition}{Proposition}[section]

\newtheorem{corollary}{Corollary}[section]

\theoremstyle{definition}
\newtheorem{definition}{Definition}[section]

\newtheorem{example}{Example}[section]

\theoremstyle{remark}
\newtheorem{remark}{Remark}[section]

\newcommand{\interior}[1]{%
	{\kern0pt#1}^{\mathrm{o}}%
}

\def\Cc{{\mathcal C}}

\def\dirac{{\partial\!\!\!/}}

\def\Tc{{\mathcal T}}
\def\eh{{{\bf{\rm{e}}}}}

\def\Fc{{\mathcal F}}
\def\Ec{{\mathcal E}}
\def\Uc{{\mathcal U}}

\def\Dc{{\mathcal D}}

\def\ka{{K_{{{\rm area}}}}}

\title[The scalar curvature in conical manifolds]{The scalar curvature in conical manifolds: some results on existence and obstructions}%

\author{Levi Lopes de Lima}
\address{Universidade Federal do  Cear\'a,
	Departamento de Matem\'atica, Campus do Pici, R. Humberto Monte, s/n, 60455-760,
	Fortaleza/CE, Brazil}
\thanks{Partially supported by CNPq/Brazil grant
	312485/2018-2 and by FUNCAP/CNPq/PRONEX grant 00068.01.00/15.}
\email{levi@mat.ufc.br}

\begin{document}

	\begin{abstract}
		We first show that existence results due to Kazdan and Warner \cite{kazdan1975existence} and Cruz and Vit\'orio \cite{cruz2019prescribing} can be extended to the category of manifolds with an isolated conical singularity. More precisely, we check that, under suitable conditions on the link manifold, any bounded and smooth function which is   negative somewhere is the scalar curvature of some conical metric (with the boundary being minimal whenever it is non-empty). By way of comparison, we complement this analysis by indicating how index theory, as developed by Albin and Gell-Redman \cite{albin2016index}, may be used to transfer to this conical setting some of the 
		classical obstructions to the existence of metrics with positive scalar curvature in the spin context. 
		In particular,  we use a version of the notion of infinite $K$-area to obstruct such metrics. 
	\end{abstract}
	
	\maketitle

	\section{Introduction and statement of the results}
	\label{intro}
	
	The general problem of prescribing the scalar curvature function in a given smooth closed manifold is a central theme in Riemannian Geometry. Since in principle bounds on this invariant have little or no direct influence at all on the large scale geometry of the underlying metric, it is expected that a huge amount of functions might be realized as the scalar curvature. In a sense this has been confirmed by Kazdan and Warner \cite{kazdan1975existence}, who showed that any function on a closed manifold of dimension $n\geq 3$ which is negative somewhere is the scalar curvature of some metric.

	The main purpose of this note is to suitably extend this result to compact manifolds carrying singularities. Very likely, the results described here hold true more generally for manifolds carrying edge-like singularities as in \cite{albin2016index,jeffres2010conformal,bahuaud2019long,   akutagawa2014yamabe}, but we restrict ourselves to the subclass of manifolds with an isolated conical singularity.

	\begin{definition}\label{iso:con:def}
		Let $X$ be a topological space which is smooth everywhere except possibly at some  $p\in X$, which has a  neighborhood $\mathcal C$  diffeomorphic to $(0,1)\times L$, where $L$ is a closed smooth manifold of dimension $d=n-1$. Additionally, 
		assume that there exists a Riemannian metric $g$ on $X_0:=X\backslash\{p\}$ such that
		with respect to conical coordinates $(x,z)\in \mathcal C$, where $x$ is a defining funtion for $\{p\}$ and $z$ is a local coordinate in $L$, it can written as 
		$g=g_0+h$, where $g_0=dx^2+x^2h_0$, with $h_0$ being a fixed metric in $L$, and $h$ is a smooth $2$-tensor satisfying $|h|_{g_0}=O(x^\lambda)$, $\lambda\geq 2$. 
		We then say that 
		$g$ is a {\em conical metric} and
		$(X,g)$ is a {\em conical manifold}. 
		Also, the manifold $(L,h_0)$ is called the {\em link}. 
	\end{definition}
	
	\begin{remark}\label{conical:type}
		An alternate way of expressing the conical metric is obtained after setting  $x=e^{-t}$, so that $g=e^{-2t}(dt^2+h_0+O(1))$, where $(t,z)$ varies in the half-cylinder $\widehat{\mathcal C}=[0,+\infty)\times L$. 
	\end{remark}

	We recall that, similarly to the existence result by Kazdan and Warner mentioned earlier, it has been proved in \cite[Theorem 1.4]{cruz2019prescribing} that any smooth function on $X$ which is negative somewhere in $X\backslash \partial X$ is the scalar curvature $\kappa_g$ of a metric $g$ with respect to which $\partial X$ is minimal; see also \cite{cruz2020critical} for a sharper result under a volume constraint. 
	This raises the question of whether similar results hold in the conical setting. Our main result provides an affirmative answer to those questions under natural assumptions. For this we consider the self-adjoint elliptic operator
	\[
	\mathcal L^n_{h_0}=-\alpha\Delta_{h_0}+\kappa_{h_0},\quad \alpha=\frac{4(n-1)}{n-2},\quad n=\dim X_0\geq 4,
	\]
	acting on $C^\infty(L)$ and let $\lambda_{h_0}$ be its lowest eigenvalue.
	
	\begin{theorem}\label{exist:quasi:pos}
		Let $(X,g)$ be a conical manifold whose link $(L,h_0)$ satisfies $\lambda_{h_0}>0$.
		Then any bounded and smooth function which is negative somewhere in $X_0\backslash \mathcal C$ is the scalar curvature of some conical metric in $X$. Also, the same result holds true if $\partial X\neq\emptyset$, in which case we additionally have that $\partial X$ is minimal with respect to the desired metric.
	\end{theorem}
	
	\begin{remark}\label{sca::obst}
		As $x\to 0$, 
		\begin{equation}\label{scal:asym}
			\kappa_g=x^{-2}\left(\kappa_{h_0}-d(d-1)\right)+O(x^{-1}),\quad d=n-1,
		\end{equation}
		so that $\kappa_{h_0}=d(d-1)$ is a necessary condition for having $\kappa_g=O(x^{-1})\in L^q(X_0,g)$, for some $q>n/2$. 
		As explained in \cite[Section 2.3]{akutagawa2014yamabe}, it suffices to assume that $\lambda_{h_0}>0$ to make sure that $\kappa_g$ satisfies this key integrability condition after possibly conformally deforming $g$ by a suitable power of the original defining function. The solution of the Yamabe problem in conical spaces \cite[Theorem 2.5]{akutagawa2014yamabe} then allows us to find a conformal conical metric $\mathfrak g$ with constant negative scalar curvature, which happens to be the first step in the proof of Theorem \ref{exist:quasi:pos}; see Section \ref{exist:scal} for more details. We note, however, that the metric $\mathfrak g$ so obtained is {\em not} necessarily quasi-isometric to the background metric $g$. In fact, this only happens if we make the much stronger requirement that $\lambda_{h_0}=d(d-1)$, in which case the preliminary conformal deformation mentioned above is not needed.  Finally, note that $\mathcal L^n_{h_0}=\mathcal L_{h_0}^{d}+\beta_n\Delta_{h_0}$, where $\beta_n>0$ and 
		\[
		\mathcal L^{d}_{h_0}=-\frac{4(d-1)}{d-2}\Delta_{h_0}+\kappa_{h_0}
		\]
		is the conformal Laplacian of $(L,h_0)$. Thus, $\lambda_{h_0}>0$ implies $\mathcal L^d_{h_0}>0$, that is, the link is Yamabe positive.
	\end{remark}

	The existence result by Kazdan and Warner should be contrasted
	with the well-known obstructions 
	describing how the
	existence of positive scalar curvature metrics interacts with the topology of the underlying manifold in the spin setting. 
	Indeed, starting with the seminal work by Lichnerowicz \cite{lichnerowicz1963spineurs}, who identified the non-vanishing of the $\widehat A$-genus of a closed spin manifold as an obstruction, 
	the machinery based on Dirac operators has been consistently used to establishing groundbreaking results in this area \cite{gromov1980spin,gromov1983positive}. More recently, Gromov 
	was able to recast some of his classical results with Lawson in terms of the notion of $K$-area, a topological invariant of manifolds inspired by $K$-theory. More precisely, infinite $K$-area obstructs metrics with positive scalar curvature in the spin setting \cite[Section 5]{gromov1996positive}.

	Our next aim is to complement the existence result in Theorem \ref{exist:quasi:pos} by indicating how the index theory in \cite{albin2016index} may be used to transfer the above mentioned obstructions to the category of manifolds carrying singularities. 
	Thus, assuming that  the underlying manifold $X$ is spin,  the choice  of  a conical metric $g$ on $X_0$ allows us to consider the associated  Dirac operator
	$\dirac:\Gamma(S_{X_0})\to \Gamma(S_{X_0})$. Here, $S_{X_0}$ denotes the spinor bundle of $X_0$,
	equipped with the Clifford product
	$\mathfrak c:\Gamma(TX_0)\to\Gamma({{\rm End}}(S_{X_0}))$. The Weitzenb\"ock
	decomposition for the corresponding Dirac Laplacian is
	\begin{equation}\label{dec}
		\dirac^2=
		{\nabla}^*\nabla+\frac{1}{4}\kappa_g,
	\end{equation}
	where ${\nabla}^*\nabla$ is the Bochner Laplacian of
	$S_{X_0}$ and $\kappa_g$ is the scalar curvature of $g$.

	It is proved in \cite[Theorem 1.1]{albin2016index} that $\dirac$, viewed as an unbounded, densely defined operator in the appropriate $L^2$ space of spinors,
	is essentially self-adjoint if
	a certain `'geometric Witt assumption'', displayed here in (\ref{geo:witt}) below, 
	holds true; see Theorem \ref{albin:mellin} for a rewording of this fundamental result in the setting of the so-called Mellin-Sobolev spaces introduced in Section \ref{exist:scal}. 
	We note that this Witt assumption
	holds whenever  $\kappa_g\geq 0$ in the conical region $\mathcal C$.
	If we assume further that $n=2k$ then 
	${\dirac}$ induces a Dirac operator $\prescript{c}{}{D}=\prescript{c}{}{D}^+\oplus \prescript{c}{}{D}^-$ on the link $(L,h_0)$; see (\ref{induced:dir}). We then consider the {\em adjusted} $\hat A$-genus of $(X,g)$:
	\begin{equation}\label{ahat:con}
		\prescript{c}{}{\widehat A}(X,g):=\int_{X_0}\widehat A(TX_0)+\int_L\mathcal T\widehat A(TX_0)-\frac{1}{2}\eta^+(0),
	\end{equation}
	where $\widehat{A}(TX_0)$ is the $\widehat{A}$-class
	of $TX_0$, $\mathcal T\widehat A(TX_0)$ is the associated transgression, $\eta^+(0)$ is the eta invariant of $\prescript{c}{}{D}^+$ (see Section \ref{con:ind:for:s}) and integration picks the top degree term in each case.
	Under these conditions, the main result in \cite{albin2016index}
	identifies this invariant to the index of the corresponding chiral Dirac operator ${\dirac}^+$. Precisely,
	\begin{equation}\label{indfor:alb:gell}
		{\rm ind}\,{\dirac}^+=\prescript{c}{}{\widehat A}(X,g).
	\end{equation} 
	Here, we specialize to the case $Y=\{p\}$ in their notation.
	The next result is just a restatement of part of \cite[Theorem 1.3]{albin2016index}; see also Corollary \ref{vanis}.

	\begin{theorem}\label{albin:top} 
		If $(X,g)$ as above is such that its scalar curvature is non-negative everywhere and positive somewhere then $\prescript{c}{}{\widehat A}(X,g)=0$. 
	\end{theorem}
	
	This result provides a {geometric} obstruction to the existence of metrics with positive scalar curvature. From our perspective, it should be viewed as an intermediate step towards an obstruction which is more topological in nature. This involves adapting the classical notion of $K$-area, originally due to Gromov \cite{gromov1996positive}, to the conical setting; see Section \ref{kar} below. The resulting invariant,  still called the $K$-area of $(X,g)$  and denoted  $K_{\rm area}(X,g)$, in principle yields a geometric invariant of the conical manifold $(X,g)$. Nevertheless, the fact that it is finite or infinite is a `'topological'' property of the conical manifold. To explain this latter point, 
	recall
	that it is natural
	to declare that the morphisms in the category of conical manifolds are diffeomorphims $\Psi:X_0\to X_0$ of {\em bounded distortion}, so that $g$ and $\Psi^*g$ are in the same quasi-isometry class.
	It turns out that having infinite $K$-area is an invariant property under such diffeomorphisms. From this viewpoint, 
	our aim here is to obstruct positive scalar curvature metrics (possibly with a mean convex boundary) in a given quasi-isometry class.

	\begin{theorem}\label{wh}
		If $X^{2k}$ is a conical spin manifold with infinite $K$-area then it does not carry a conical metric with (strictly) positive scalar curvature (in the given quasi-isometry class). Also, the same result
		holds true in case $\dim X$ is odd if we assume instead that the $K$-area
		of $X\times \mathbb T$, the product of $X$ with a circle, is infinite.
	\end{theorem}

	Theorem \ref{wh} must be assessed from the perspective that, as explained in \cite[Section 5]{gromov1996positive}, the Dirac machinery may be applied to show that, in the smooth category,  infinite $K$-area obstructs metrics with positive scalar curvature in the spin setting. This remains true 
	for metrics 
	with positive scalar curvature in the interior and non-negative mean curvature along the boundary, as depicted in the recent preprint by B\"ar and Hanke \cite[Theorem 19]{baer2020boundary}. Alternatively, this latter result may be retrieved  as a consequence of Gromov's theorem in the closed case mentioned above combined with a doubling construction which goes back to Gromov and Lawson \cite[Theorem 5.7]{gromov1980spin} and Almeida \cite[Theorem 1.1]{almeida1985minimal}; this approach is emphasized in \cite[Section 1.1]{gromov2018metric} and discussed in  \cite[Remark 35]{baer2020boundary}.
	In our setting, both approaches may be adapted to yield a proof of the corresponding obstruction in the presence of a boundary.

	\begin{theorem}\label{wh2}
		If  a spin conical manifold $X^{2k}$ with boundary  has  infinite $K$-area then it does not carry a metric with (strictly) positive scalar curvature and mean convex boundary (in the given quasi-isometry  class). 
		Also, the same result
		holds true in case $\dim X$ is odd if we assume instead that the $K$-area
		of $X\times  \mathbb T$ is infinite.
	\end{theorem}
	
	Examples of conical manifolds to which Theorems \ref{wh} and \ref{wh2} apply are discussed in Example \ref{exam:karea} below.

	This note is organized as follows. Theorem \ref{exist:quasi:pos} is established in Section \ref{exist:scal}	and its proof  relies on the 
	solution of the Yamabe problem in conical spaces \cite{akutagawa2003yamabe,akutagawa2014yamabe},
	as well as on the 
	so-called cone elliptic theory  \cite{schulze1998boundary,schrohe1999ellipticity,schrohe2005resolvent}, which is employed to study the mapping properties of a certain linearized operator whose principal part is the conical Laplacian (these analytical requisites are reviewed in Appendix \ref{con:mapp:p}). 
	The  proofs of Theorems \ref{wh} and \ref{wh2} are presented in Section \ref{demon}. This uses the 
	concept of $K$-area, which  is reviewed in Section \ref{kar}, 
	as well as the natural extensions to the conical setting of classical twisted index formulas presented in Section \ref{con:ind:for:s}. 
	This latter material is  based on the `'incomplete edge'' index theory developed in \cite{albin2016index}.
	Another key ingredient is 
	the rounding construction due to Gromov-Lawson and Almeida, which is discussed in Appendix \ref{bend}. 
	
	\vspace{0.3cm}
	
	\noindent
	{\bf Acknowledgments.} I thank S. Almaraz and C.T. Cruz for conversations. I also thank an anonymous referee for carefully reading the text and thus substantially contributing to its final presentation.
	
	\section{An existence result}\label{exist:scal}
	
	Here we prove Theorem \ref{exist:quasi:pos}, 
	which provides an extension to the conical setting of the existence results 
	appearing in \cite{kazdan1975existence,cruz2019prescribing}. 
	In the following  we always assume that $n\geq 4$.

	If $[g]_c$ is a conformal class of conical metrics we define the corresponding {\em Yamabe invariant} as 
	\begin{equation}\label{yamabe:1}
		Y^c([g]_c)=\inf_{0\neq u\in H^{1,2}(X,g)}\frac{\int_{X_0}\left(\alpha|\nabla_g u|^2+\kappa_gu^2\right)d{\rm vol}_{ g}}{\left(\int_{X_0}u^{2n/(n-2)}d{\rm vol}_{ g}\right)^{\frac{n-2}{n}}}, \quad \alpha=\frac{4(n-1)}{n-2},
	\end{equation}
	where $H^{1,2}(X,g)$ is the standard Sobolev space (i.e. the closure of $C^\infty_{\rm cpt}(X_0)$ with respect to the Sobolev norm).
	Also, if $X$ carries an inner compact boundary $\partial X$ with $\partial X\cap\mathcal C=\emptyset$ then we set instead
	\begin{equation}\label{yamabe:2}
		Y_{\partial}^c([g]_c)=\inf_{0\neq u\in H^{1,2}(X,g)}\frac{\int_{X_0}\left(\alpha|\nabla_g u|^2+\kappa_gu^2\right)d{\rm vol}_{ g}+2\int_{\partial X}H_g u^2 d{\rm vol}_{g|_{\partial X}}}{\left(\int_{X_0}u^{2n/(n-2)}d{\rm vol}_{ g}\right)^{\frac{n-2}{n}}},
	\end{equation}
	where $H_{g}$ is the mean curvature of $\partial X$. In each case we assume that the corresponding  invariant is finite. 
	
	To proceed with the proof, let us assume initially that $\partial X=\emptyset$.
	The first step  is to find a {conformal} conical metric $\mathfrak g=u_0^{{4}/{(n-2)}}g$ with {constant} negative scalar curvature.
	As indicated in (\ref{scal:asym}), the existence of a conformal metric to $g$ with {\em bounded} (not necessarily constant) scalar curvature 
	places   geometric obstructions on the original link metric $h_0$.
	Nevertheless, if 
	$\lambda_{h_0}>0$  
	then
	these obstructions  are met and  
	the	existence of such a  metric $\mathfrak g$ with $\kappa_{\mathfrak g}=-1$ is assured. Indeed, as in \cite[Subsection 4.32]{besse2007einstein}, we may inject a sufficiently large amount of negative scalar curvature around some point in $X_0\backslash \mathcal C$ so as to make the background metric $g$ to satisfy $Y^c([g]_c)<0$. On the other hand, if  $Y_L$ is the Yamabe invariant of the conformal class of the 
	product metric $dt^2+h_0$ on
	the
	cylinder $\mathbb R\times L$, it is known that $\lambda_{h_0}>0$ implies $Y_L>0$ \cite[Corollary 2.3]{akutagawa2014yamabe}.  
	Thus,  $Y^c([g]_c)<Y_L$. With this ``Aubin-type'' inequality at hand, the discussion surrounding \cite[Theorem 2.5]{akutagawa2014yamabe} assures the existence of a conformal conical metric $\mathfrak g$ as above; see also \cite[Theorem A (a)]{akutagawa2003yamabe}, where the connection with the cylindrical representation of the conical region in Remark \ref{conical:type} is explored.
	
	\begin{remark}\label{vert:jeffres}
		If the background metric satisfies $\kappa_g<0$ everywhere, existence results for the conical conformal metric $\mathfrak g$ as above have been obtained in \cite{bahuaud2014yamabe,bahuaud2019long}, by means of the long-term convergence of the corresponding normalized Yamabe flow, and in \cite[Theorem 1]{jeffres2010conformal}, where a  barrier construction is used. 
	\end{remark}

	We next try to implement the perturbation argument in \cite{kazdan1975existence} in this conical setting. 
	We consider the map 
	\begin{equation}\label{mapping:Au}
		A(u)=-u^{-a}\left(\alpha\Delta_{\mathfrak g}u+u\right), \quad a=\frac{n+2}{n-2}.
	\end{equation}
	Due to the conformal transformation properties of the scalar curvature, $A(u)$ is the scalar curvature of the conformal metric $u^{4/(n-2)}\mathfrak g$.
	If $\bf 1$ denotes the function identically equal to $1$ then $A({\bf 1})=-{\bf 1}=\kappa_{\mathfrak g}$ and our aim now is to solve the equation $A(u)=f$ for $f$ close to $-{\bf 1}$.
	Clearly, this involves studying the mapping properties of the linearization of $A$ at $u={\bf 1}$,
	\begin{equation}\label{oper:brac}
		\dot A_{\bf 1}=-\alpha\Delta_{\mathfrak g}+b, \quad b={a-1}=\frac{4}{n-2}, 
	\end{equation}
	which we do in a scale of Sobolev spaces which takes into account the conical singularity. 
	
	Given  $\gamma\in \mathbb R$ and integers $k\geq 0$ and $p\geq 1$, we define $\mathcal H_\gamma^{k,p}(X,g)$ to be the space of all distributions $u\in L^p_{\rm loc}(X_0)$, $X_0=X\backslash \{p\}$,   such that:
	\begin{itemize}
		\item for any cutoff function $\varphi$ with $\varphi\equiv 1$ near $p$ and $\varphi\equiv 0$ outside $\mathcal C$, we have that $(1-\varphi)u$ lies in the standard Sobolev space $H^{k,p}(X_0,g)$;
		\item there holds
		\[
		x^{\frac{n}{2}-\gamma}\mathsf D^j\partial_z^\alpha (\varphi u)(x,z)\in L^p(d_+xd{\rm vol}_h), \quad j+|\alpha|\leq k,\quad d_+x=x^{-1}dx.
		\]
		Here, $\mathsf D=x\partial_x$ is the Fuchs operator. 
	\end{itemize}
	Using duality and interpolation, we may define $\mathcal H_\gamma^{s,p}(X,g)$ for any $s\in\mathbb R$. As usual, $\mathcal H_\gamma^{s,p}(X,g)$ is naturally a Banach space which is Hilbert for $p=2$.  These are the weighted Mellin-Sobolev spaces considered in \cite{schrohe1999ellipticity}. 
	We set $\mathcal H_\gamma^{\infty,p}(X,g)=\cap_{s}\mathcal H_\gamma^{s,p}(X,g)$.
	For further reference we recall the relevant embedding theorem; see \cite[Remark 2.2]{coriasco2007realizations} and 	\cite[Corollary 2.5]{roidos2013cahn}.
	
	\begin{proposition}\label{sob:emb}
		One has a continuous embedding $\mathcal H_{\gamma'}^{s,p}(X,g)\hookrightarrow \mathcal H_\gamma^{s,p}(X,g)$ if $\gamma'\geq\gamma$. Also,
		if $s>n/p$ then any $u\in \mathcal H_\gamma^{s,p}(X,g)$ is continuous in $X_0$ and satisfies $u(x)=O(x^{\gamma-n/2})$ as $x\to 0$.
	\end{proposition}
	
	It is clear that the operator $\dot A_{\bf 1}$ is elliptic and defines a bounded map
	\begin{equation}\label{lap:cont:sob}
		\dot A_{\bf 1}:\mathcal H_\gamma^{s,p}(X,g)\to \mathcal H_{\gamma-2}^{s-2,p}(X,g) 
	\end{equation} 
	However, a key point in the analysis of an elliptic operator in a conical manifold is that, differently from what happens in the smooth case, invertibility of its principal symbol does not suffice to make sure that a parametrix exists. In particular, it is not clear whether (\ref{lap:cont:sob}) is Fredholm for some value of the weight $\gamma$. It turns out that this Fredholmness property and the associated index are insensitive to $(s,p)$ but depend crucially on $\gamma$ \cite{schrohe1999ellipticity}. Indeed, it follows from general principles that whenever this map is Fredholm for some value of $\gamma$ then it remains Fredholm for all but a discrete set of values of $\gamma$, with the index possibly jumping only when $\gamma$ reaches these exceptional values.  
	We now state a useful result that confirms this expectation for the map (\ref{lap:cont:sob}).
	A proof is included in Appendix \ref{con:mapp:p}.
	
	\begin{theorem}\label{self:adj}
		The map $\dot A_{\bf 1}$ in (\ref{lap:cont:sob}) is Fredholm of index $0$ 
		for
		$\gamma\in((4-n)/2,n/2)$ if $n\geq 5$ and $\gamma\in[0,2)$ if $n=4$.
		Also, the same holds true if we further impose Neumann boundary conditions along a boundary $\partial X\neq \emptyset$ located outside the conical region. 
	\end{theorem}
	
	\begin{corollary}\label{self:adj:cor}
		$\dot A_{\bf 1}$ as above is invertible if $0\leq \gamma<n/2$. 
	\end{corollary}
	
	\begin{proof}
		It suffices to show that $\dot A_{\bf 1}$ is injective. By Proposition \ref{sob:emb}, we may assume that $\gamma=0$. Take $w\in {\rm ker}\,\dot A_{\bf 1}\cap\mathcal H^{s,p}_0(X,g)$, so that  $\alpha\Delta_{\mathfrak g}w=bw$. 
		By Theorem \ref{map:p:lap} and Remark \ref{Fried:d=3}, integration by parts does not yield a contribution coming from the singularity. 
		We thus obtain 
		\[
		-\alpha\int_X|\nabla_{\mathfrak g}w|^2=b\int_X|w|^2,
		\] 
		a contradiction unless $w=0$.
	\end{proof}

	With these results at hand, we now turn to the adaptation of the perturbation argument in \cite{kazdan1975existence}. 
	First note that any bounded measurable  function lies in $\mathcal H^{0,p}_\gamma(X,g)$ because we are assuming that $\gamma<n/2$. In particular, ${\bf 1}\in \mathcal H^{0,p}_\gamma(X,g)$.  By the Inverse Function Theorem, there exists  $\varepsilon >0$ such that for any  $f\in \mathcal H_\gamma^{0,p}(X,g)$ with $\|f+{\bf 1}\|_{\mathcal H_\gamma^{0,p}}<\varepsilon$ there exists $u\in \mathcal H_\gamma^{2,p}(X,g)$ such that $A(u)=f$. 
	Moreover, by shrinking $\varepsilon$ we may assume that $\|u-{\bf 1}\|_{\mathcal H_\gamma^{2,p}}$ is correspondingly small.
	If $f$ is smooth and $p$ is chosen large enough then (interior) elliptic regularity together with Proposition \ref{sob:emb} imply, via a standard bootstrap argument, that $u\in C^\infty(X_0)\cap \mathcal H_\gamma^{\infty,p}(X,g)$. Hence, $u>0$ and $\widetilde g=u^{{4}/{(n-2)}}\mathfrak g$ is a conformal metric with $\kappa_{\widetilde g}=f$. Proposition \ref{sob:emb} implies that $u=1+O(x^{\gamma-n/2})=O(x^{-(n-2)/2+\epsilon})$ if $\gamma\leq 1+\epsilon$. Thus, we may apply \cite[Proposition 3.1]{akutagawa2014yamabe} to conclude that $u=c+o(1)$, $c>0$, so that  $\widetilde g$ is a conical metric.

	We now appeal to the next result, whose proof is a straightforward adaptation of the argument leading to \cite[Theorem 2.1]{kazdan1975existence}. Note that, by the previous reasoning, we are assuming that $0\leq \gamma\leq 1+\epsilon$.
	
	\begin{proposition}
		Let $\phi,\phi'\in C^0(X_0)\cap L^\infty(X_0)\hookrightarrow \mathcal H_\gamma^{0,p}(X,g)$. Then, $\min \phi <\phi' <\max \phi$ if and only if for any $\varepsilon>0$ there exists a diffeomormism $\Psi:X_0\to X_0$ of bounded distortion (in particular, preserving the quasi-isometry class of conical metrics)
		such that $\|\phi\circ \Psi-\phi'\|_{\mathcal H_\gamma^{0,p}(X,g)}<\varepsilon$.
	\end{proposition}
	
	Thus, if $\widetilde f$ is bounded, smooth and negative somewhere as in Theorem \ref{exist:quasi:pos}, there exists $K>0$ such that $K\min \widetilde f < -1 <K\max \widetilde f$, and the proposition above applies to assure the existence of $\Psi$ such that $\|K\widetilde f\circ\Psi+{\bf 1}\|_{\mathcal H_\gamma^{0,p}}<\varepsilon$. By the previous argument with $f=K\widetilde f\circ\Psi$, there exists a conical metric $\widetilde g$ such that $\kappa_{\widetilde g}=K\widetilde f\circ\Psi$. Thus, $g_1=K^{1/2}(\Psi^{-1})^*\widetilde g$
	is the conical metric whose scalar curvature is $\widetilde f$. This completes the proof of Theorem \ref{exist:quasi:pos} as long as $\partial X=\emptyset$.
	
	To carry out the proof in case $\partial X\neq \emptyset$ we observe that, under a conformal deformation  $\mathfrak g=u_0^{4/(n-2)}g$, 
	the mean curvature of $\partial X$ transforms as
	\begin{equation}\label{conf:bd}
		H_{\mathfrak g}=u_0^{-n/{(n-2)}}\left(\frac{2(n-1)}{n-2}\frac{\partial u_0}{\partial \nu}+H_{g}u_0\right),
	\end{equation} 
	where  $\nu$ is the outward unit normal vector field.
	Of course, we may assume that 
	$H_g=0$ and, as before, that $Y_\partial^c([g]_c)<0<Y_L$ (here, we should use the half-cylinder $[0,+\infty)\times L$  to define $Y_L$). This allows us 
	to adapt the arguments in \cite{akutagawa2003yamabe,akutagawa2014yamabe} in order to find a minimizer $u_0$ for the Yamabe quotient in the right-hand side of (\ref{yamabe:2}). As this minimizer obviously satisfies the Neumann condition along $\partial X$, we obtain a  
	conformal metric $\mathfrak g=u_0^{4/(n-2)}g$ with $\kappa_{\mathfrak g}=-1$ and {\em minimal} boundary. From this point on, we may repeat the perturbation argument above  
	by using the invertibility in the last assertion of Theorem \ref{self:adj} to complete the proof of Theorem \ref{exist:quasi:pos}.  
	
	\begin{remark}\label{bound:geo:c}
		The passage from $g$ to $\widetilde g$ above preserves the quasi-isometry class of the metric if we additionally assume that 
		$\lambda_{h_0}=d(d-1)$; compare with Remark \ref{sca::obst}.
	\end{remark}

	\section{$K$-area}
	\label{kar}
	
	The concept of $K$-area was introduced by Gromov  in order
	to quantify previous results on obstructions to the existence of
	metrics with positive scalar curvature in the smooth case \cite{gromov1980spin}. In this section we briefly review this
	classical notion and then check how it may be easily adapted to the conical setting. 
	
	\subsection{The classical $K$-area}\label{kar:class}
	Let $(X^{2k},g)$ be a closed Riemannian manifold (not necessarily
	spin). By pulling
	back a suitable hermitian bundle over the unit sphere $\mathbb S^{2k}$ under
	a degree one map, we see that the set of complex vector bundles
	over $X$ which are {\em homologically non-trivial} (i.e. with
	at least a nonzero Chern number) is nonempty. Notice that by
	Chern-Weil theory  the Chern numbers, which are topological
	invariants of $\Ec$, can be computed by integrating over $X$ certain
	universal differential forms depending on the curvature tensor
	$R^{\nabla}$ of any compatible connection $\nabla$ on $\Ec$. Thus
	$\Ec$ is {\em homologically trivial} (i.e. all Chern numbers vanish)
	if $R^{\nabla}=0$.
	
	We then let $\Ec$ vary over the set of homologically non-trivial
	hermitian bundles (and compatible connections) over $X$ and define
	the $K$-area of $(X,g)$ by
	\begin{equation}\label{karea}
		\ka(X,g)=\sup\frac{1}{\|R^{\nabla}\|_{g}},
	\end{equation}
	where
	\[
	\left\|R^{\nabla}\right\|_g=\sup_{v\wedge w\neq 0} \frac{\|R^{\nabla}_{v,
			w}\|_{\rm op}}{\,\,\|v\wedge w\|_g},
	\]
	$\|\,\,\|_{\rm op}$ is the operator norm on fiber endomorphisms 
	and $\|v\wedge w\|^2_g=g(v,v)g(w,w)-g(v,w)^2$.  Hence, the
	$K$-area as defined above is a Riemannian invariant but the fact
	that it is finite or infinite turns out to be a topological property of $X$. We note that the use of the operator norm in the definition allows us to conclude that if $\Ec_1$ and $\Ec_2$ are hermitian bundles over $(X,g)$ with compactible connections $\nabla_1$  and $\nabla_2$, respectively, then
	\begin{equation}\label{oper:cons}
		\left\|R^{\nabla_1\oplus\nabla_2}\right\|_g=\max\left\{\left\|R^{\nabla_1}\right\|_g,\left\|R^{\nabla_2}\right\|_g\right\},\quad 	\left\|R^{\nabla_1\otimes\nabla_2}\right\|_g=\left\|R^{\nabla_1}\right\|_g+\left\|R^{\nabla_2}\right\|_g, 	
	\end{equation}
	where $\nabla_1\oplus\nabla_2$ and $\nabla_1\otimes\nabla_2$ are the naturally induced connections on $\Ec_1\oplus\Ec_2$ and $\Ec_1\otimes\Ec_2$, respectively \cite{davaux21k}.

	If $X$ carries a boundary $\partial X$ we retain the definition
	(\ref{karea}) but restricting to bundles which are trivial in a neighborhood of $\partial X$.  The allowable connections are required to
	be flat in this neighborhood so that characteristic numbers related to $\Ec$
	are obtained by integrating over $X$ characteristic differential
	forms with compact support in the interior of $X$.
	For the sake of briefness, bundles
	meeting these conditions will be called {\it admissible}. 
	By the
	pullback construction referred to above, one always has $\ka(X,g)>0$.

	Examples of closed manifold with infinite $K$-area include  
	tori and, more generally, finitely enlargeable manifolds \cite{gromov1980spin}. Among these, we may cite solvable manifolds and non-positively curved manifolds whose fundamental group is residually finite. 
	A large class of examples in the bordered case appears in \cite[Lemma 16]{baer2020boundary}.
	We also note the following
	characterization, which is readily derived from the definition.
	
	\begin{proposition}\label{grom}
		$\ka(X,g)=+\infty$ if and only if for any $\epsilon>0$ there exists
		a homologically non-trivial bundle $(\Ec,\nabla)$ over $X$ with
		$\|R^{\nabla}\|_g\leq\epsilon$.
	\end{proposition}

	\subsection{The $K$-area of conical manifolds}\label{kar:kc}
	We now discuss the appropriate adaptation of the notion of $K$-area in (\ref{karea}) to the conical setting. This is achieved by simply requiring that an admissible bundle $\Ec$ should be trivial when restricted to the conical region $\mathcal C$. 
	Of course, we also assume that $\Ec|_\Cc$ is endowed with a flat connection.
	We then call such  a bundle $\mathcal C$-admissible.  Hence, if we take the supremum as in (\ref{karea}) but now restricting to {\em homologically non-trivial} $\mathcal C$-admissible bundles, 
	this yields a geometric invariant, still denoted $K(X,g)$, and named the  $K$-area of $(X,g)$. It is clear, however, that the fact that the $K$-area so defined is finite or infinite is  a {\em quasi-isometric} property of the conical manifold (it is invariant under diffeomorphisms of bounded distortion). 
	
	This notion may be easily adapted to the case in which the conical manifold carries a non-empty boundary $\Sigma$ with $\Sigma\cap\mathcal C=\emptyset$. As in the smooth case, we  retain (\ref{karea}) but restrict to $\mathcal C$-admissible, homologically non-trivial bundles which are trivial in a neighborhood of $\Sigma$, where they are endowed with a flat connection. As before, the fact that the $K_{\rm area}$ so defined is finite or infinite remains being a quasi-isometric property of the underlying conical manifold.

	\begin{remark}\label{karea:edge}
		This notion of $K$-{area} may be straightforwardly extended to the case in which the underlying manifold carries a single singular stratum, say $Y$.  This is needed in the proofs of Theorems \ref{wh} and \ref{wh2} when ${\dim X}$ is odd, in which case $Y=\{p\}\times\mathbb T\hookrightarrow X\times\mathbb T$.   
	\end{remark}
	
	\begin{example}\label{exam:karea}
		Take any compact smooth manifold  $W$ with infinite $K$-area (possibly with a non-empty boundary $\partial W$) and consider an open, relatively compact subset $U\subset W$ with smooth boundary $L$ which is
		$K$-negligible in the sense of \cite[Definition 14]{baer2020boundary}. Now excise $U$ and replace it by a conical region with link $L$ (after possibly slightly perturbing the manifold around the link so as to restore smoothness). It is easy to check that the resulting conical manifold has infinite $K$-area.
	\end{example}

	\section{The twisted conical index formulas}\label{con:ind:for:s}

	Here we present the natural extensions of classical index formulas to the conical setting in the presence of a $\Cc$-admissible bundle $\Ec\to X$ as in Subsection \ref{kar:kc}. 
	These are key ingredients in the proofs of Theorem \ref{wh}
	and \ref{wh2}.
	
	Let us assume initially that $X$ carries no boundary. Thus, if $\Ec$ is $\mathcal C$-admissible we can consider the corresponding twisted Dirac operator $\dirac_{\Ec}$. 
	
	\begin{proposition}\label{albin:fff}
		If $\kappa_g|_{\mathcal C}\geq 0$ then
		\begin{equation}\label{ext:ind:f}
			{\rm ind}\,{\dirac}_\Ec^+
			= \prescript{c}{}{\widehat A}_\Ec(X,g),
		\end{equation}
		where,  analogously to (\ref{ahat:con}),
		\begin{equation}\label{ahat:con:2}
			\prescript{c}{}{\widehat A}_\Ec(X,g):=\int_{X_0}\widehat A(TX_0)\wedge {\rm ch}(\Ec)+r\int_L\mathcal T\widehat A(TX_0)-\frac{r}{2}\eta^+(0), \quad r={\rm rank}\,\Ec.
		\end{equation}
	\end{proposition}
	
	\begin{proof}
		The key point here is that the $\mathcal C$-admissibility of $\Ec$ clearly implies that   ${\dirac}_{\Ec}|_{\mathcal C}$ is just the direct sum of $r={\rm rank}\,\Ec$ copies of ${\dirac}|_{\mathcal C}$. It follows that $\mathcal R^{[\nabla]}|_{\mathcal C}=0$ in (\ref{dec:twist}) below and hence the ``geometric Witt assumption'' in \cite{albin2016index} still holds true for ${\dirac}_{\Ec}$ as long as $\kappa_g|_{\mathcal C}\geq 0$. In particular, $\dirac_\Ec$ is essentially self-adjoint and we may consider the corresponding chiral twisted Dirac operators $\dirac^\pm_\Ec$, which are adjoint to each other; {for self-adjointness, see also Theorem \ref{albin:mellin} below}. Thus, the main result in \cite{albin2016index} immediately implies that the Fredholm index of $\dirac^+_\Ec$ is given by (\ref{ext:ind:f}).
	\end{proof}

	We now present a version of the celebrated Atiyah-Patodi-Singer index formula in our context. 
	Thus, we  take $X$ to be a conical spin manifold with
	dimension $n=2k$ and a {non-empty} smooth boundary
	$\Sigma\subset X$ with $\Sigma\cap\mathcal C=\emptyset$. Introduce Fermi coordinates
	$(y,u)\in \Sigma\times [0,\delta)\to \Uc$ in a collar neighborhood
	$\Uc$ of $\Sigma$ and set $\Sigma_u=\{(y,u); y\in\Sigma\}$ so that
	$\Sigma_0=\Sigma$. Then, restricted to $\Uc$,
	\begin{equation}\label{dir:fermi}
		\dirac=\mathfrak c\left({\partial_u}\right)\left({\partial_u}+ D-\frac{1}{2}H\right),
	\end{equation}
	where  $H$ is the mean curvature of the embeddings 
	$\Sigma_u\subset\Uc$ {(computed with respect to the inward unit vector field)} and  $D$ is the tangential Dirac operator, which is
	defined as follows. For each $u$, $S_X|_{\Sigma_u}$ comes equipped
	with the Clifford product $\mathfrak c^u=-\mathfrak c(\partial_u)\mathfrak c$.  Hence,  if we  define the induced connection by
	$$
	\nabla^u=\nabla-\frac{1}{2}\mathfrak c^u(B),
	$$
	where $B$ is the shape operator of the embedding $\Sigma_u\hookrightarrow
	\Uc$, then
	$$
	D=\sum_{i=1}^{2k-1}\mathfrak c^u(\eh_i)\nabla^u_{\eh_i},
	$$
	where $\{\eh_i\}$ is an orthonormal basis tangent to  $\Sigma_u$.
	
	After tensoring with $\Ec$, a $\Cc$-admissible bundle as above, we obtain a  first order self-adjoint
	elliptic operator $D_\Ec$ acting on
	$\Gamma(S_{X_0}\otimes\Ec|_{\Sigma})$ and commuting with the chirality operator,  
	so we can decompose
	$\mathbb{S}_\Ec:=S_{X_0}\otimes \Ec|_{\Sigma}$ as $\mathbb{S}_\Ec=\mathbb{S}_\Ec^+\oplus\mathbb{S}_\Ec^-$,
	and accordingly, $D_\Ec=D_\Ec^+\oplus D_\Ec^-$, with
	$D_\Ec^{\pm}$ being self-adjoint. Under the natural identification
	$\mathbb{S}_\Ec^+=\mathbb{S}_\Ec^-$ one has
	$D_\Ec^+=-D_\Ec^-$ and hence ${{\rm Spec}}(D_\Ec)$ is
	symmetric with respect to $0\in\mathbb{R}$, but of course this does
	not need happen with the factors $D_\Ec^{\pm}$. Thus, for ${{\rm
			Re}}\,z\gg 0$ we define the eta function
	\[
	\eta_\Ec^+(z)=\sum_{0\neq\lambda\in {{\rm Spec}}(D_\Ec^+)} ({\rm
		sign}\,\lambda)|\lambda|^{-z}.
	\] 
	This extends meromorphically to the whole
	complex plane with the origin not being a pole and $\eta_\Ec^+(0)$
	is a well defined real number called the {\it eta invariant} of
	$D_\Ec^+$. It measures the overall asymmetry of ${{\rm
			Spec}}(D_\Ec^+)$ with respect to the origin. Henceforth, we denote by $\eta^+(0)$ the corresponding {\em untwisted} spectral invariant of $D^+$.

	We now introduce the so-called Atiyah-Patodi-Singer (APS) type boundary conditions.
	If $\Dc$ is a self
	adjoint elliptic operator acting on sections of a bundle
	$\Fc\to\Sigma$, we denote by $\Pi_{I}(\Dc):L^2(\Fc)\to L^2(\Fc)$ the
	spectral projection of $\Dc$ associated to the interval
	$I\subset\mathbb{R}$. Also, if $\psi\in\Gamma(S_W\otimes\Ec)$ we set
	${\varphi}=\psi|_{\Sigma}$. Now  consider $\Gamma_{\geq
		0}(S^+_{X_0}\otimes\Ec)=\{\psi\in
	\Gamma_{\rm cpt}(S^+_{}X_0\otimes\Ec);\Pi_{[0,+\infty)}(D_\Ec)\varphi=0\}$
	and $\Gamma_{
		>0}(S^-_{X_0}\otimes\Ec)=\{\psi\in \Gamma_{\rm cpt}(S^-_{X_0}\otimes\Ec);
	\Pi_{(0,+\infty)}(D_\Ec)\varphi=0\}$, where the subscript ${\rm cpt}$ means compactly supported sections. These are  the (core) domains of
	the operators
	\[
	\dirac^+_{\Ec,\geq 0}={{\dirac^+}|}_{\Gamma_{\geq
			0}(S^+_{X_0}\otimes\Ec)}:\Gamma_{\geq 0}(S^+_{X_0}\otimes \Ec)\to
	\Gamma(S^-_{X_0}\otimes\Ec) \label{dec1}
	\]
	and
	\[
	\dirac^-_{\Ec,> 0}={{\dirac^-}|}_{\Gamma_{>
			0}(S^-_{X_0}\otimes\Ec)}:\Gamma_{>0}(S^-_{X_0}\otimes\Ec)\to
	\Gamma(S^+_{X_0}\otimes\Ec), \label{dec2}
	\]
	respectively. Since $\mathfrak c(\partial_u)$ anti-commutes with $D$, these operators are adjoint to each other if the `'geometric Witt assumption'' holds; see Remark \ref{vanis:bd}. Moreover,
	$\dirac^+_{\Ec,\geq 0}$ is a Fredholm operator with a well defined
	index
	\[
	{{\rm ind}}\,\dirac^+_{\Ec,\geq 0}=\dim\ker \dirac^+_{\Ec,\geq
		0}-\dim\ker \dirac^-_{\Ec,> 0}.
	\]
	The following formula, which is the natural extension to the conical setting of a classical result \cite{atiyah1975spectral,booss1993bavnbek,gilkey1993index,grubb1992heat}, computes this invariant.

	\begin{theorem}\label{aps:conic} Under the conditions above,  
		\begin{equation}\label{indfor:con}
			{\rm ind}\,\prescript{c}{}{\dirac}^+_{\Ec,\geq 0}
			= \prescript{c}{}{\widehat A}_\Ec(X_0,g)
			+\int_{\Sigma}\Tc\widehat{A}(TX_0)\wedge {\rm ch}(\Ec)-\xi_\Ec^+(0),
		\end{equation}
		where 
		\begin{equation}\label{eta:twist}
			\xi_\Ec^+(0)=\frac{1}{2}\left(\eta_\Ec^+(0)+\dim\ker
			D_\Ec^+\right).
		\end{equation}
	\end{theorem}
	
	\begin{proof} This is just a matter of 
		splicing together the parametrices in \cite{albin2016index,grubb1992heat} and then using  standard heat asymptotics. The details are omitted.
	\end{proof}

	\begin{remark}\label{rig:con:chou}
		In case the conical metric is {\em rigid} ($h=0$), versions of the index formulas in (\ref{indfor:alb:gell}) and (\ref{ext:ind:f}) appear in \cite{chou1985dirac}. We also refer to \cite{lesch1997differential}, where similar formulas may be found. 
	\end{remark}	
	
	\section{The proofs of Theorems \ref{wh} and \ref{wh2}}\label{demon}
	
	Here we present the proofs of Theorems  \ref{wh} and \ref{wh2}. Given the index formulas (\ref{ext:ind:f}) and (\ref{indfor:con}), they are rather straightforward extensions of the classical arguments. 
	
	\subsection{The proof of Theorem \ref{wh}}\label{proof:wh} Assume that $(X,g)$ is a conical manifold with $\kappa_g>0$ and $K_{\rm area}(X)=+\infty$. We first observe that (\ref{ext:ind:f}) may be rearranged as 
	\[
	{\rm ind}\,{\dirac}_{\widetilde \Ec}^+=r\prescript{c}{}{\widehat A}(X,g)+\int_{X_0}\widehat A(TX)\wedge\widehat{{\rm ch}}(\Ec),
	\]
	where the {\em reduced} Chern character of $\Ec$ is
	\[
	{\widehat{\rm ch}}(\Ec)={\rm ch}(\Ec)-r
	={\rm ch}_1(\Ec)+{\rm ch}_2(\Ec)+\ldots,
	\]
	with ${\rm ch}_i(\Ec)\in \Gamma(\wedge^{2i}TX_0)$  defined by a universal homogeneous characteristic polynomial of degree $i$   in  $R^{\nabla}$. Since by Theorem \ref{albin:top} we may assume that $\prescript{c}{}{\widehat A}(X,g)=0$, this reduces to 
	\begin{equation}\label{ind:f:red:ch}
		{\rm ind}\,{\dirac}_{\Ec}^+=\int_{X_0}\widehat A(TX)\wedge\widehat{{\rm ch}}(\Ec),
	\end{equation}

	We now observe that, by the analogue of Proposition \ref{grom}, for each $\epsilon>0$ there exists a $\Cc$-admissible, homologically non-trivial bundle $(\Ec,\nabla)$ such that $\|R^{\nabla}\|_g\leq\epsilon$. It then follows that $\|\mathcal R^{[\nabla]}\|\leq (2k(2k-1)/2)\epsilon$, where $\mathcal R^{[\nabla]}$ appears in the Weitzenb\"ock decomposition of the corresponding Dirac Laplacian:
	\begin{equation}\label{dec:twist}
		\dirac_\Ec^2=\nabla^*\nabla+\frac{1}{4}\kappa_g+\mathcal R^{[\nabla]};  
	\end{equation}
	see \cite[Lemma 1]{baer2020boundary}.
	Notice that (\ref{dec:twist}) is the twisted version of (\ref{dec}).
	
	Thus, if $\epsilon$ is small enough we have $\kappa_g/4+\mathcal R^{[\nabla]}>0$, so that $\ker \dirac_\Ec=\{0\}$ and hence ${\rm ind}\,{\dirac}_{\Ec}^+=0$. From (\ref{ind:f:red:ch}),
	\begin{equation}\label{sub:est:r:ad:ad}
		\int_{X_0}\widehat A(TX)\wedge{\widehat{\rm ch}}(\Ec)=0,
	\end{equation}
	and we are in a position to 
	follow \cite{gromov1996positive} and bring the mechanism of Adams operations to our discussion. 
	Recall that this is a rule that to each $\mu\in \mathbb N$  and $\Ec$ as above associates a (virtual) bundle $\Psi_\mu\Ec$  
	which is a universal expression in terms of tensor products of exterior powers of $\Ec$. It is compatible with the Chern character map in the sense that 
	\[
	{\rm ch}(\Psi_\mu\Ec)=\sum_{j\geq 0}{\rm ch}_j(\Ec)\mu^j, 
	\]
	In particular, ${\rm rank}\,\Psi_\mu\Ec=r={\rm rank}\,\Ec$.
	Moreover, for each $\nu\in\mathbb N$ and a multi-index $\mu_{(\nu)}=(\mu_1,\ldots,\mu_\nu)$, one has that 
	\[
	\Psi_{\mu_{(\nu)}}\Ec:=\Psi_{\mu_1}\Ec\otimes\cdots\otimes \Psi_{\mu_\nu}\Ec
	\] 
	satisfies ${\rm rank}\,\Psi_{\mu_{(\nu)}}\Ec=r^\nu$ and 
	\begin{equation}\label{prod:adams}
		{\rm ch}_j\Psi_{\mu_{(\nu)}}\Ec=\sum_{i_1+\ldots+i_\nu=j}\mu_1^{i_1}\cdot\ldots\cdots \mu_\nu^{i_\nu}{\rm ch}_{i_1}(\Ec)\wedge\ldots\wedge {\rm ch}_{i_\nu}(\Ec). 
	\end{equation}
	Notice that each $\Psi_{\mu_{(\nu)}}\Ec$ is $\mathcal C$-admissible by (\ref{oper:cons}).

	We now observe that, again   by (\ref{oper:cons}),  the bounds on $R^{\nabla}$ are `'stable'' under Adams operations, so 
	the argument leading to (\ref{sub:est:r:ad:ad}) works fine for {any}  $\Psi_{\mu_{(\nu)}}\widetilde\Ec$ as long as $\nu$ remains bounded. Thus,  
	\begin{equation}\label{sub:est:r:ad:ad2}
		{\rm ind}\,\dirac_{	\Psi_{\mu_{(\nu)}}\Ec}^+=	\int_{X_0}\widehat A(TX)\wedge{\widehat{\rm ch}}(\Psi_{\mu_{(\nu)}}\Ec)=0. 
	\end{equation}
	On the other hand, from (\ref{prod:adams}) with $\nu=k$, we have
	$
	{\rm ind}\,\dirac_{	\Psi_{\mu_{(k)}}\Ec}^+=\mathsf Q(\mu_{(k)})
	$,
	a polynomial of degree at most $k$ given by 
	\begin{equation}\label{poly:adms}
		\mathsf Q(\mu_{(k)})=\sum_{i_1+\ldots+i_k=k}\mu_1^{i_1}\cdot\ldots\cdot \mu_k^{i_k}\int_X{\rm ch}_{i_1}(\Ec)\wedge\ldots\wedge {\rm ch}_{i_k}(\Ec)+\ldots,
	\end{equation}
	where the dots correspond to `'lower order terms'' (that is, terms whose coefficients are characteristic numbers necessarily involving the curvature of $g$). 
	It follows that the top degree coefficients in (\ref{poly:adms}), the Chern character numbers of $\Ec$,  vanish. Since the Chern numbers of $\Ec$ may be expressed as a rational linear combination  in these coefficients, we conclude that $\Ec$ is homologically trivial. This contradiction completes the proof of Theorem \ref{wh}  at least if ${\rm dim}\,X$ is even. 
	
	In the odd case, one has to deal with $X\times\mathbb T$ endowed with the product metric, a manifold which is singular along an edge diffeomorphic to $\mathbb T$. Index formulas for such objects also follow from the general theory in \cite{albin2016index} and the discussion in Section \ref{con:ind:for:s} above. In fact, since a tubular neighborhood of $\mathbb T$ is isometric to the product $\mathcal C\times\mathbb T$,  the resulting formulas are the same as  in (\ref{ext:ind:f}) and (\ref{indfor:con}) if we normalize the circle to have unit length. 
	Moreover, as observed in Remark \ref{karea:edge}, the notion of $K$-{area} may be extended to this case. 
	Hence, this situation may be handled with a straightforward adaptation of the argument above to this simple edge setting. We leave the details to the interested reader.

	\subsection{The proof of Theorem \ref{wh2}}\label{proof:wh2}

	In accordance with the general strategy put forward in \cite[Section 1.1]{gromov2018metric}, the idea is to reduce to Theorem \ref{wh} via a doubling argument.  We then consider the doubled manifold $\mathsf DX= X\cup_{\partial X} - X$, where the minus sign denotes opposite orientation. Note that $\mathsf D X$ is spin. Moreover, if $\partial X\hookrightarrow X$ is mean convex, it follows from \cite{gromov1980spin, almeida1985minimal} that the doubled metric $\mathsf Dg$ on $\mathsf D X$ may be deformed to a conical metric, say $\widetilde g$, with positive scalar curvature everywhere and which additionally remains in the same quasi-isometry class as $\mathsf D g$. Precisely, the deformation may be chosen so as to meet the following properties:
	\begin{itemize}
		\item it is supported in an arbitrarily small neighborhood of the common inner boundary $\partial X$;
		\item it may be taken arbitrarily small in the $C^0$-topology;
	\end{itemize}
	Granted these properties of the deformation, which for the sake of completeness we discuss in Appendix \ref{bend} below, we see that $K_{\rm area}(\mathsf DX,\widetilde g)=+\infty$ and Theorem \ref{wh} immediately applies to complete the argument.
	
	\begin{remark}\label{bar:han}
		We may also obtain a proof of Theorem \ref{wh2} by directly applying APS index theory to the bordered conical manifold $(X,g)$. Indeed, if $\Ec$ is a $\Cc$-admissible bundle, recall that there exists a neighborhood $U$ 
		of $\Sigma$ such that $\Ec|_U=U\times\mathbb C^r$ endowed with a flat connection. Thus,
		(\ref{indfor:con}) gets reduced to 
		\[
		{\rm ind}\,{\dirac}^+_{\Ec,\geq 0}
		= \prescript{c}{}{\widehat A}_\Ec(X_0,g)
		+r\int_{\Sigma}\Tc\widehat{A}(TX_0)-r\xi^+(0),
		\]
		where $\xi^+(0)$ is the untwisted version of (\ref{eta:twist}). Combining this with (\ref{ahat:con:2}) and using that 
		\[
		\int_{X_0}\widehat A(TX_0)\wedge{\rm ch}(\Ec)=r\int_{X_0}\widehat A(TX_0)+ 	\int_{X_0}\widehat A(TX_0)\wedge{\widehat{\rm ch}}(\Ec),
		\]
		we find that 
		\begin{eqnarray*}
			{\rm ind}\,{\dirac}^+_{\Ec,\geq 0}
			& = & 
			r\left(\widehat A(X_0,g)+\int_\Sigma\mathcal T\widehat A(X_0)-\xi^+(0)\right)+\int_{X_0}\widehat A(TX_0)\wedge{\widehat{\rm ch}}(\Ec)\\
			& = & r\,{\rm ind}\,{\dirac}^+_{\geq 0}+\int_{X_0}\widehat A(TX_0)\wedge{\widehat{\rm ch}}(\Ec),
		\end{eqnarray*}
		where we used the untwisted version of (\ref{indfor:con}) in the last step. This may be rewritten as 
		\[
		{\rm ind}\,{\dirac}^+_{\widehat\Ec,\geq 0}=\int_{X_0}\widehat A(TX_0)\wedge{\widehat{\rm ch}}(\Ec),
		\]
		where $\widehat\Ec=\Ec-\Theta^r$ is the associated virtual bundle (here, $\Theta^r=X_0\times\mathbb C^r$). With this `'relative'' APS index formula at hand, the argument follows exactly as in \cite[Section 2.3]{baer2020boundary}. 
	\end{remark}

	\appendix
	\section{The deformation}\label{bend}
	
	Here we discuss the claims regarding the deformation  used in the proof of Theorem \ref{wh2} above.  
	The argument appeared originally in \cite{almeida1985minimal,gromov1980spin}. 
	This also follows from the more sophisticated procedure in \cite{baer2020boundary}.
	
	We start by fixing Fermi coordinates in a collar neighborhood $\mathcal V=\partial X\times [0,\delta)$ of  $\partial X$, so that the metric $g$ may be written as $g=h_v+dv^2$, where $h_v=g|_{\partial X_v}$ with $\partial X_v=\{v\}\times\partial X$.  Recall that the second fundamental form $B_v$ of $\partial X_v\hookrightarrow X$ is given by 
	\[
	B_v=-\frac{1}{2}\frac{\partial h_v}{\partial v}. 
	\]  
	Following \cite{almeida1985minimal}, we take a everywhere positive $f:[0.+\infty)\to\mathbb R$ and consider the metric 
	$\widehat g=\widehat h_v+dv^2$, where $\widehat h_v=fh_v$. With obvious notation, 
	\[
	\widehat B_v=-\frac{1}{2}\frac{\partial (fh_v)}{\partial v}=-\frac{1}{2}f'h_v+fB_v, \quad \widehat H_v=-\frac{n-1}{2}\frac{f'}{f}+H_v,
	\] 
	where $H$ denotes mean curvature and the prime is derivative with respect to $v$. We now choose $f$ so that $f'(0)<0$ and  $f(v)=1$ for $v\geq \delta/2$. Since $H_0\geq 0$, we have $\widehat H_0>0$. 
	Moreover, if we further impose that 
	$\|f-1\|_{C^2}<\delta$ and  
	pick
	$\delta$ small enough then the scalar curvature of $\widehat g$ remains negative.
	It is clear that the deformation $\widehat g -g$ meets all the requirements posed in the proof of Theorem \ref{wh2} (in fact, at this stage it may be chosen arbitrarily small in the $C^2$ topology).  
	
	We now pass to the last step of the deformation, starting with the manifold $(X,\widehat g)$ constructed above. We follow \cite{gromov1980spin} and consider a collar neighborhood $\mathcal W\hookrightarrow X$ of $\partial X$ with radius $10\varepsilon<\delta$ small and radial parameter $w\in[0,10\varepsilon)$.   Within $\mathcal W\times \mathbb R$ endowed with the product metric, say $G$,  we consider $\mathcal Z_{\varepsilon}$, the set of points whose distance to $\mathcal W_{\geq \varepsilon}=\{x\in \mathcal W; {\rm dist}_{\widehat g}(x,\partial X)\geq \varepsilon\}$ equals $\varepsilon$. Thus, 
	$\mathcal Z_{\varepsilon}$ consists of two $\varepsilon$-equidistant copies, namely $\mathcal W^\pm_{\geq \varepsilon}:=\mathcal W_{\geq \varepsilon}\times\{\pm\varepsilon\}$, of $\mathcal W_{\geq \varepsilon}$ and a semi-circular piece $\mathcal Z^\bullet_{\varepsilon}$ which is isometric to the product of $\partial X$ with the semi-circle of radius $\varepsilon$ parametrized by $\theta\in [-\pi/2,\pi/2]$; see Figure \ref{figg} below. 
	
	\begin{figure}[h]
		\begin{center}
			\begin{tikzpicture}
				\begin{scope}
					\draw[<->] (-4.5,1) -- (-4.5,0);
					\draw[->] (-4.5,1) -- (-4.5,2);
					\draw[<->] (-4.5,0) -- (-4.5,-1);
					\draw (-4.5,-1) -- (-4.5,-2);
					\draw (-6,1) -- (0,1);
					\draw[dashed] (-7,0) -- (0,0);
					\draw[thick] (0,0) -- (1,0);
					\draw[thick] (0,0) -- (0.7,0.7);	
					\draw (-6,-1) -- (0,-1);
					\draw[very thick] (0,-1) arc(-90:90:1cm);
					\draw (0.3,0) arc(0:16:1cm);
					\node at (-4.3,-0.5) [shape=circle] {$\varepsilon$};
					\node at (-4.3,0.5) [shape=circle] {$\varepsilon$};
					\node at (-7.3,0) [shape=circle] {$\mathcal W$};
					\node at (1.12,-0.7)  [shape=circle] {$\mathcal Z_{\epsilon}^\bullet$};
					\node at (1.4,0)  [shape=circle] {$\partial X$};
					\node at (-0.4,0.3)  [shape=circle] {$\partial\mathcal W_{\geq \varepsilon}$};
					\node at (-3,0.3)  [shape=circle] {$\mathcal W_{\geq \varepsilon}$};
					\node at (-3,1.3)  [shape=circle] {$\mathcal W^+_{\geq \varepsilon}$};
					\node at (-3,-1.3)  [shape=circle] {$\mathcal W^-_{\geq \varepsilon}$};
					\node at (0,0)  [shape=circle] {$\bullet$};
					\node at (1,0)  [shape=circle] {$\bullet$};
					\node at (0.5,0.2)  [shape=circle] {$\theta$};
				\end{scope}
			\end{tikzpicture}
		\end{center}
		\caption{The Gromov-Lawson construction}
		\label{figg}
	\end{figure}
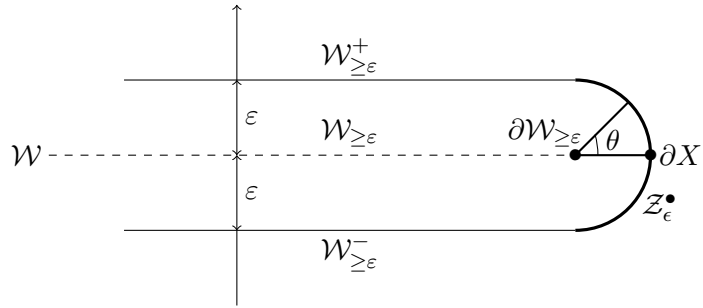 
	
	As explained in \cite{gromov1980spin,gromov2018metric}, as $\varepsilon\to 0$ we obtain asymptotic expansions for the principal curvatures $\lambda_1,\dots,\lambda_n$ of $\mathcal Z^\bullet_{\varepsilon}$ in terms of the principal curvatures $\mu_1,\dots,\mu_{n-1}$ of $\partial X$ and $\cos\theta$. Precisely, 
	\[
	\lambda_i=\left(\mu_i+O(\varepsilon)\right)\cos\theta + o(\varepsilon), \quad \lambda_n=\varepsilon^{-1}+O(1). 
	\]
	From this we find that 
	\[
	\kappa_{G|_{\mathcal Z^\bullet_{\varepsilon}}}(x,\theta)=\kappa_{\widehat g}(x)+\left(2\varepsilon^{-1}\widehat H_0+O(1)\right)\cos\theta+o(1), \quad x\in \partial X,
	\]
	and since
	$\widehat H_0>0$ and $\kappa_{\widehat g}>0$, we conclude that $\kappa_{G|_{\mathcal Z^\bullet_{\varepsilon}}}>0$ if $\varepsilon$ is small enough.
	Thus, we may view $(\mathcal Z^\bullet_{\varepsilon},G|_{\mathcal Z^\bullet_{\varepsilon}})$ as interpolating between two copies of $(X,\widehat g)$ so as to form the doubled manifold $(\mathsf D X, \widetilde g)$ appearing in the proof of Theorem \ref{wh2} above. Finally, note that the  creases along which $\mathcal Z^\bullet_{\varepsilon}$ meets the copies $\mathcal W^\pm_{\geq \varepsilon}$  
	may be easily smoothed out while still preserving all the relevant geometric properties.
	This justifies the existence of the desirable deformation. 
	
	\section{Mapping properties in weighted Mellin-Sobolev spaces}\label{con:mapp:p}

	Our aim here is to present the proof of Theorem \ref{self:adj}. This is certainly well-known and may be confirmed in a variety of ways on inspection of standard sources; see for instance \cite{mazzeo1991elliptic,schulze1998boundary,lesch1997differential}, among others. However, since in these references the specific result we are interested in is usually immersed in rather elaborate theories, we find it convenient to include here a sketch of the argument in the setting of the weighted Mellin-Sobolev spaces introduced in Section \ref{exist:scal}. We further illustrate the flexibility of these methods by  establishing the mapping properties of the conical Dirac operator, thus providing,  in the conical setting,  an alternate approach to some of the analytical results in \cite{albin2016index}.

	\subsection{The Laplacian and its perturbations}\label{lap:map:sub}
	As a preliminary step, we consider the mapping properties of the Laplacian
	\begin{equation}\label{map:lap:g}
		\Delta_g:\mathcal H_\gamma^{s,p}(X,g)\to  \mathcal H^{s-2,p}_{\gamma-2}(X,g),
	\end{equation}
	where $g$ is a conical metric. Recall that $g=g_0+h$, where $g_0=dx^2+x^2h_0$ and $|h|_{g_0}=O(x^2)$. 
	A possible route to approach the Fredholmness
	of this map is to consider instead the {\em core} Laplacian
	\begin{equation}\label{core:lap}
		(\Delta_g,C^\infty_{\rm cpt}(X_0)):C^\infty_{\rm cpt}(X_0)\subset \mathcal H_\gamma^{0,2}(X,g)\to \mathcal H_\gamma^{0,2}(X,g),
	\end{equation}
	an unbounded, densely defined operator. As usual, the closure of this operator is $(\Delta_g,D_{\rm min}(\Delta_g)$, where 
	\[
	D_{\rm min}(\Delta_g)=\left\{u\in \mathcal H_\gamma^{0,2}(X,g); \exists \{u_n\}\subset C^\infty_c(X_0); u_n\stackrel{\mathcal H_\gamma^{0,2}}{\to} u, \{\Delta_gu_n\}\,\, {\rm is}\,\, \mathcal H_\gamma^{0,2}-{\rm Cauchy} \right\}.
	\]
	On the other hand, the associated adjoint operator is $(\Delta_g,D_{\rm max}(\Delta_g))$, where 
	\[
	D_{\rm max}(\Delta_g)=\left\{u\in \mathcal H_\gamma^{0,2}(X,g);\Delta_gu\in \mathcal H_\gamma^{0,2}(X,g)\right\}.
	\] 
	Regarding these notions, the following facts are well-known:
	
	\begin{itemize}
		\item $D_{\rm min}(\Delta_g)\subset D_{\rm max}(\Delta_g)$;
		\item If $(\widehat\Delta_g,{\rm Dom}(\widehat\Delta_g))$ is a closed extension of $(\Delta_g,C^\infty_c(X_0))$ then
		\[
		D_{\rm min}(\Delta_g)\subset {\rm Dom}(\widehat\Delta_g)\subset  D_{\rm max}(\Delta_g).
		\]  
	\end{itemize}
	Thus, in order to understand the set of closed extensions, we need to look at the subspaces of  	
	the {\em asymptotics space} 
	\[
	\mathcal Q(\Delta_g):=\frac{D_{\rm max}(\Delta_g)}{D_{\rm min}(\Delta_g)}.
	\]
	In particular, $\mathcal Q(\Delta_g)=\{0\}$ implies that the Laplacian {has a unique closed extension and hence the associated map (\ref{map:lap:g}) is Fredholm. In particular, it is essentially self-adjoint whenever it is symmetric.}

	Recall that in conical coordinates $(x,z)$ around $p$ and in terms of the Fuchs operator $\mathsf D=x\partial_x$, the Laplacian $\Delta_g$ satisfies 
	\[
	P:=x^2\Delta_g=\mathsf D^2+\left(d-1+O(x)\right)\mathsf D+\Delta_{h_0}.  
	\]
	As already highlighted, the needed ingredients to establish the mapping properties for $\Delta_g$ include not only its ellipticity, but also the invertibility of the so-called {\em conormal symbol}, which is obtained by freezing the coefficients of $P$ at $x=0$, that is, passing to 
	\[
	P_0=\mathsf D^2+\left(d-1\right)\mathsf D+\Delta_{h_0}, 
	\] 
	and then applying the Mellin transform $\mathsf M$; see \cite{schrohe1999ellipticity,schulze1998boundary} and also (\ref{conor:symb}) below,  where this construction is actually applied to an appropriate conjugation of $P_0$.
	For our purposes, it suffices to know that  $\mathsf M$
	is a linear map that to each well-behaved function $f:\mathbb R_+\to\mathbb C$ associates another function $\mathsf M(f):U_f\subset \mathbb C\to \mathbb C$
	meeting the following properties: 
	\begin{itemize} 
		\item For each $\theta\in\mathbb R$, the map 
		\[
		x^\theta L^2(\mathbb R_+,d_+x)\stackrel{\mathsf M}{\longrightarrow} 
		L^2(\Gamma_{-\theta}), 
		\]
		is an isometry (recall that $d_+x=x^{-1}dx$). Here, {$\Gamma_\alpha=\{\zeta\in\mathbb C;{\rm Re}\,\zeta=\alpha\}$, $\alpha\in\mathbb R$, and  $x^\theta L^2(\mathbb R_+,d_+x)$ is endowed with the inner product
			\[
			\langle u,v\rangle_{x^\theta L^2(\mathbb R_+,d_+x}=\langle x^{-\theta}u,x^{-\theta}v\rangle_{L^2(\mathbb R_+,d_+x)}.
			\]
		}
		Moreover, each element $u$ in the image extends holomorphically to the half-space $\{\zeta\in\mathbb C;{\rm Re}\,\zeta{>}-\theta\}$ (Notation: $u\in\mathscr H(\{{\rm Re}\,\zeta>-\theta\})$).
		\item $\mathsf M(\mathsf Df)(\zeta)=-\zeta\mathsf M(f)(\zeta)$.  
	\end{itemize}

	This suggests to work with the `'Mellin'' volume element $d{\rm vol}_{g,c}=d_+x\sqrt{\det \mathfrak h}dz$ instead of the original volume element $d{\rm vol}_{g}=x^ddx\sqrt{\det \mathfrak h}dz$. Here, $\mathfrak h=h_0+O(x^{\lambda-2})$ and $z$ is a local coordinate in $L$.
	This is implemented by working `'downstairs'' in the  diagram below, where $\tau=x^{\frac{d+1}{2}}$ is unitary and $\Delta^\tau_g=\tau\Delta_g\tau^{-1}$:
	\begin{equation}\label{diag}
		\begin{array}{ccc}
			D_{\rm max}(\Delta_g)\subset \mathcal H_\gamma^{0,2}(X,g) & \xrightarrow{\,\,\,\,\Delta_g\,\,\,\,} & 	\mathcal H_\gamma^{0,2}(X,g) \\
			\tau	\Big\downarrow &  & \Big\downarrow \tau \\
			D_{\rm max}(\Delta^\tau_g)\subset	x^\gamma	L^2(d{\rm vol}_{g,c}) & \xrightarrow{\,\,\,\,\Delta^\tau_g\,\,\,\,} & x^\gamma	L^2(d{\rm vol}_{g,c}) 
		\end{array}
	\end{equation}
	
	{
		\begin{remark}\label{symm}
			It is immediate to check that, near the singularity, 
			\[
			\langle \Delta_g u,v\rangle_{\mathcal H^{0,2}(X,g)}=\int x^{-2\gamma}v\Delta_g u\,d{\rm vol}_g,
			\]
			so the horizontal maps in (\ref{diag}) define symmetric operators if and only if $\gamma=0$. Notice also that the same conclusion holds for any operator which is formally self-adjoint with respect to $d{\rm vol}_g$.
		\end{remark}
	}
	
	Let $u\in D_{\rm max}(\Delta_g)$. Thus, $v:=\tau u\in D_{\rm max}(\Delta^\tau_g)$ satisfies $x^{-\gamma}v\in L^2(d{\rm vol}_{g,c})$, so that 
	$\mathsf M(v)\in{\mathscr H}(\{{\rm Re}\,\zeta> -\gamma\})$. On the other hand, if 
	\[
	P_0^\tau:=\tau P_0\tau^{-1}=
	\mathsf D^2-2\mathsf D-\frac{(d+1)(d-3)}{4}+\Delta_{h_0}, 
	\]
	then $w:=P_0^\tau v$ satisfies $x^{-2-\gamma}w=x^{-\gamma}\tau\Delta_gu\in L^2(d{\rm vol}_{g,c})$, so that $\mathsf M(w)\in{\mathscr H}(\{{\rm Re}\,\zeta> -2-\gamma\})$.
	By taking Mellin transform,
	\[
	\mathsf M(w)(\zeta,z)=\sigma_{\Delta^\tau_g}(\zeta)\mathsf M(v)(\zeta,z), 
	\]
	where 
	\begin{equation}\label{conor:symb}
		\sigma_{\Delta^\tau_g}(\zeta)=	\zeta^2+2\zeta-\frac{(d+1)(d-3)}{4}+\Delta_{h_0}
	\end{equation}
	is the conormal symbol of $\Delta^\tau_g$.
	The conclusion is that, at least formally, 
	\begin{equation}\label{formally}
		\mathsf M(v)(\zeta,z)=\sigma_{\Delta^\tau_g}^{-1}(\zeta)\mathsf M(w)(\zeta,z), 
	\end{equation}
	but we should properly handle the poles of $\sigma_{\Delta^\tau_g}^{-1}$
	located within the critical strip $\Gamma_{-2-\gamma,-\gamma}$, which we may gather together in the {\em asymptotics set}
	\[
	\Lambda^\tau_{\gamma}:=\left\{\zeta\in\mathbb C;\zeta^2+2\zeta-\frac{(d+1)(d-3)}{4}-\mu=0,\mu\in{\rm Spec}(\Delta_{h_0})\right\}\cap \Gamma_{-2-\gamma,-\gamma}.
	\]
	Here,  $\Gamma_{c,c'}=\{\zeta\in \mathbb C; c<{\rm Re}\,\zeta<c'\}$ for $c<c'$. 
	Alternatively, we may consider  
	\[
	\widetilde\Lambda^{\tau,\pm}_{\gamma}=\left\{\mu\in{\rm Spec}(\Delta_{h_0});\delta_{d,\mu}^{\pm}\in\Gamma_{-1-\gamma,1-\gamma}\right\},
	\]	
	where
	\[
	\delta^\pm_{d,\mu}=\pm\frac{1}{2}\sqrt{(d-1)^2+4\mu}.
	\]
	After applying Mellin inversion to (\ref{formally}) and using the appropriate pseudo-differential calculus \cite{schrohe1999ellipticity,schulze1998boundary}, we obtain
	\begin{equation}\label{expansion}
		v-w=\sum_{\mu\in\widetilde\Lambda^{\tau,+}_{\gamma}}a^+_\mu(z) x^{1-\delta_{d,\mu}^+}+
		\sum_{\mu\in\widetilde\Lambda^{\tau,-}_{\gamma}}
		a^-_\mu(z) x^{1-\delta_{d,\mu}^-},
	\end{equation}	
	where $a^{\pm}_\mu\in\ker(\Delta_h+\mu)$.
	Notice that the terms in the expansion all lie in $\ker P_0^\tau$. Here we are assuming that $d\geq 3$ as this guarantees that $|\delta^\pm_{\mu,d}|\geq (d-1)/2>0$, which by its turn prevents the appearance of log-terms in the expansion; see \cite[Lemma 2.2]{mazzeo2012analytic} for more on this point, but be aware that our normalization for the exponents is slightly different from theirs. In particular, we obtain a quite explicit description of the relevant asymptotics  space:
	\begin{equation}\label{decom:cauchy}
		\mathcal Q(\Delta_g^{\tau})=\oplus_{\mu\in\widetilde\Lambda^{\tau,+}_{\gamma}}E_{\mu}^+ \bigoplus \oplus_{\mu\in\widetilde\Lambda^{\tau,-}_{\gamma}} E_{\mu}^-, \quad 	E_{\mu}^\pm=\langle a^\pm_\mu x^{1-\delta_{d,\mu}^\pm} \rangle.
	\end{equation}	
	As mentioned above, subspaces of this space correspond under conjugation (i.e. `'upstairs'') to the domains of closed extensions of the core Laplacian. Moreover, any such closed extension, say $(\widehat\Delta_g,{\rm Dom}(\widehat\Delta_g))$, is Fredholm with index given by
	\begin{equation}\label{ind:form:fred}
		{\rm ind}\,(\widehat\Delta_g,{\rm Dom}(\widehat\Delta_g))={\rm ind}\,(\Delta_ g,D_{\rm min}(\Delta_g) -\dim\,\frac{{\rm Dom}(\widehat\Delta_g)}{D_{\rm min}(\Delta_g)}.
	\end{equation}
	In any case, from the theory above we easily derive the 
	next result. 
	
	\begin{theorem}\label{map:p:lap}
		The core Laplacian {has a unique closed extension} whenever  $\widetilde\Lambda^{\tau,\pm}_{\gamma}=\emptyset$. In particular, this happens for $\gamma=0$ and $d\geq 3$, {in which case it is essentially self-adjoint (and hence  the corresponding map is Fredholm of index $0$).}
	\end{theorem}

	\begin{remark}\label{Fried:d=3}
		The last assertion follows by combining Remark \ref{symm} with the fact that  $d\geq 3$ implies $|\delta^{\pm}_{d,\mu}|\geq 1$ for any $\mu\geq 0$, so that $\widetilde\Lambda^{\tau,\pm}_{0}=\emptyset$ indeed; compare with \cite[Theorem 2.1]{mooers1999heat} and \cite[Example 3.5]{coriasco2002differential}. Thus, any $v\in D_{\rm max}(\Delta_g)\subset \mathcal H^{0,2}_0(X,g)$ actually lies in $D_{\rm min}(\Delta_g)=D_{\rm max}(\Delta_g)\cap\mathcal H^{2,2}_{2^-}(X,g)$ \cite[Proposition 2.3]{schrohe2005resolvent}. Of course, this should be viewed as a regularity result. Note also that, since $g=dx^2+x^2\mathfrak h_x$, $\mathfrak h_x=h_0+O(x^{\lambda-2})$, an application of the co-area formula leads to 
		\[
		\lim_{x\to 0}\int_{\{x\}\times L}|\mathsf D^j\partial^\alpha v|^2d{\rm vol}_{\mathfrak h_x}=0,\quad j+|\alpha|\leq 2,
		\]
		for any 	
		$v\in  \mathcal H^{0,2}_0(X,g)$. This clearly implies that integration by parts involving $v\in D_{\rm max}(\Delta_g)$ does not yield a contribution coming from the singularity, which is consistent with Theorem \ref{map:p:lap}. 
	\end{remark}
	
	The next result extends the previously established Fredholmness to a suitable range of weights.
	
	\begin{theorem}\label{map:p:lap:cor}
		The Laplacian mapping (\ref{map:lap:g}) is Fredholm of index $0$ for $\gamma\in((4-n)/2,n/2)$ if $n\geq 5$ and $\gamma\in[0,2)$ if $n=4$.
	\end{theorem} 
	
	\begin{proof}
		As already observed, Fredholmness and the associated index do not depend on the pair $(s,p)$ but only on $\gamma$. The key point now is that, as it is manifest from (\ref{ind:form:fred}), the strategy to preserve Fredholmness as $\gamma$ varies involves precluding the crossing of  poles of $\sigma_{\Delta_g^\tau}^{-1}$ through the critical line $\Gamma_{-\gamma}$. Precisely, if we set
		\[
		\Xi^\tau_{\gamma}:=\left\{\zeta\in\mathbb C;\zeta^2+2\zeta-\frac{(d+1)(d-3)}{4}-\mu=0, \mu\in{\rm Spec}(\Delta_{h_0})\right\}\cap\Gamma_{-\gamma},
		\] 	
		the relevant result is that $\Delta_g$ remains Fredholm with the {\em same} index as long as $\Xi^\tau_{\gamma}=\emptyset$; see  \cite[Section 3]{schrohe1999ellipticity} or \cite[Subsection 2.4.3]{schulze1998boundary}. Certainly, this is the case for all $\gamma\in((4-n)/2, n/2)$. Since this interval always contains $\gamma=0$ for $n\geq 5$, the result follows in this case from Theorem \ref{map:p:lap}. For $n=4$, although we initially must choose $\gamma\in (0,2)$, we can extend this to include $\gamma=0$ again in view of Theorem \ref{map:p:lap}.
	\end{proof}

	We now turn to the mapping properties of $A=-\alpha\Delta_{g}+b$; compare with (\ref{oper:brac}). Since
	$x^2 A=-\alpha P+O(x^2)$,
	we see that the zero order term in $A$ does not play a role in determining the conormal symbol. Thus, $\sigma_{{A}^\tau}=-\alpha\sigma_{\Delta^\tau_g}$ and we may use the argument above to conclude the first  part of Theorem \ref{self:adj}. 
	To handle the remaining case, we recall that imposing Neumann boundary condition on $\partial X$ does not affect the index. This may be directly checked for $\Delta_g$ and $\gamma=0$ by observing that, by Theorem \ref{map:p:lap} and Remark \ref{Fried:d=3}, integration by parts does not pick a contribution from the singularity. Hence, an application of Green's formula gives that both the  kernel and the cokernel of $\Delta_g$ are formed by constant functions, so the index vanishes. It follows that $A$ is Fredholm with index $0$ as well.  
	
	\begin{remark}\label{upstairs}
		The authors in \cite{schrohe1999ellipticity} and \cite{schulze1998boundary} work `'upstairs'' in respect to the diagram (\ref{diag}), that is, before applying the conjugation $\tau=x^{n/2}$, so instead of $\Xi^\tau_{\gamma}$ they consider 
		\[
		\Xi_{\gamma}:=\left\{\zeta\in\mathbb C;\zeta^2-(d-1)\zeta-\mu=0, \mu\in{\rm Spec}(\Delta_{h_0})\right\}\cap\Gamma_{\frac{n}{2}-\gamma}.
		\] 	
		Notice that, as expected, the polynomial equation here is the Mellin transform of $P_0$ whereas the critical line is shifted to the right by $n/2$. It is immediate to check that both approaches produce the same numerical ranges for Fredholmness.
	\end{remark}
	
	\subsection{The Dirac operator}\label{dir:map:sub}
	We now explain  how the theory above may be adapted to establish the mapping properties of the conical Dirac operator
	\begin{equation}\label{dirac:map}
		{\dirac}:\mathcal H^{s,p}_\gamma(S_{X_0})\to \mathcal H^{s-1,p}_{\gamma-1}(S_{X_0})
	\end{equation}
	in the appropriate scale of Mellin-Sobolev spaces. We assume that $\partial X=\emptyset$. As usual, we first consider the core Dirac operator
	\begin{equation}\label{dirac:map:c}
		({\dirac},\Gamma_{\rm cpt}(S_{X_0})):\Gamma_{\rm cpt}(S_{X_0})\subset \mathcal H^{0,2}_\gamma(S_{X_0})\to H^{0,2}_\gamma(S_{X_0}),
	\end{equation}
	and our aim is to give conditions on $\gamma$ to make sure that the associated asymptotics space is trivial. 
	
	It follows from the general computations in \cite[Lemma 2.2]{albin2016index} that, in the conical region,
	\begin{equation}\label{induced:dir}
		{\dirac}={\mathfrak c}(\partial_x)\left(\partial_x+\frac{d}{2x}+\frac{1}{x}\prescript{c}{}{D}\right)+O(1),
	\end{equation}
	where $\mathfrak c$ is Clifford product and $\prescript{c}{}{D}$ is the Dirac operator of the spin manifold $(L,h_0)$; this also follows more directly from (\ref{dir:fermi}) if we appeal to the conformal transformation properties of the mean curvature as in (\ref{conf:bd}) and of the Dirac operator \cite[Proposition 2.31]{bourguignon2015spinorial}. Hence,
	\[
	\mathscr P:=x{\dirac}={\mathfrak c}(\partial_x)\mathscr P_0+O(x), \quad \mathscr P_0=\mathsf D+\frac{d}{2}+\prescript{c}{}{D}.
	\]
	By working `'downstairs'', we get
	\[
	\mathscr P_0^\tau:=\tau\mathscr P_0\tau^{-1}=\mathsf D-\frac{1}{2}+ \prescript{c}{}{D},
	\] 
	and after Mellin transforming this we see that the corresponding asymptotics  set is 
	\[
	\Theta^\tau_\gamma:=\left\{\zeta\in \mathbb C; \zeta+\frac{1}{2}-\vartheta=0,\vartheta\in{\rm Spec}(\prescript{c}{}{D})\right\}\cap\Gamma_{-1-\gamma,-\gamma}.  
	\]
	By arguing exactly as above, we easily obtain the following result, which should be compared with \cite[Theorem 1.1]{albin2016index}.
	
	\begin{theorem}\label{albin:mellin}
		The core Dirac (\ref{dirac:map:c}) {has a unique closed extension} whenever $\Theta^\tau_\gamma=\emptyset$. In particular, this happens for $\gamma=0$ 
		if the `'geometric Witt assumption''
		\begin{equation}\label{geo:witt}
			{\rm Spec}(\prescript{c}{}{D})\cap\left(-\frac{1}{2},\frac{1}{2}\right)=\emptyset
		\end{equation}
		is satisfied. In this case, the Dirac map (\ref{dirac:map})
		is Fredholm of index $0$ if $0\leq \gamma<1$, {with self-adjointness taking place for $\gamma=0$}. 
	\end{theorem}

	It turns out that non-negative scalar curvature in the conical region implies (\ref{geo:witt}) \cite[Theorem 1.3]{albin2016index} and hence ${\rm ind}\,{\dirac}=0$ by the previous result. But of course this says nothing about the index of the chiral Dirac operator ${\dirac}^+$, which is the invariant that appears prominently in applications. As in the smooth case, in order to obtain the corresponding vanishing result we need some more control on the scalar curvature. 
	
	\begin{corollary}\label{vanis}
		If the scalar curvature is non-negative everywhere and positive somewhere then (\ref{dirac:map}) is invertible if $0\leq\gamma <1$. In particular, ${\rm ind}\,{\dirac}^+=0$.
	\end{corollary}
	
	\begin{proof}
		It suffices to check that ${\dirac}$ is injective and of course we may assume that $\gamma=0$. By Theorem \ref{albin:mellin}, integration by parts involving 
		$\psi\in {\rm ker}\,{\dirac}$
		is insensitive to the singularity. 
		This immediately yields the integral version of (\ref{dec}), namely, 
		\begin{equation}\label{weit:tw:int}
			\int_{X_0}\left(|\nabla\psi|^2+\frac{\kappa_g}{4}|\psi|^2\right)=\int_\Sigma\left(\langle
			D\varphi,\varphi\rangle-\frac{H}{2}|\varphi|^2\right), \quad \varphi=\psi|_\Sigma,
		\end{equation}
		where $\Sigma=\partial X$.
		Since we are assuming that $\Sigma=\emptyset$, this reduces to 
		\[
		\int_{X_0}\left(|\nabla\psi|^2+\frac{\kappa_g}{4}|\psi|^2\right)=0,
		\]
		and the result follows by means of the usual vanishing argument.
	\end{proof}

	\begin{remark}\label{vanis:bd}
		We  obtain similar results if suitable boundary conditions are imposed along a non-empty  boundary $\Sigma=\partial X$ located outside the conical region. The most commonly used here are the APS discussed in Section \ref{con:ind:for:s} and the MIT bag \cite[Section 6.2]{de2020heat}, as any spinors $\psi$ and $\psi'$  meeting one of these conditions necessarily satisfy $\langle \,\prescript{c}{}{D}\psi,\psi'\rangle=0$.  In particular, 
		the corresponding realizations of $\dirac$ are self-adjoint. In the former case, it follows that the chiral Dirac operators $\dirac^+_{\Ec,\geq 0}$ and $\dirac^-_{\Ec,>0}$ are adjoint to each other. Also,
		by (\ref{weit:tw:int}) we obtain a vanishing result for $\ker \dirac$ if we additionally assume that $\partial X$ is mean convex ($H\geq 0$).   
	\end{remark}

	\bibliographystyle{alpha}
	\bibliography{karea-sing-comp-arxiv}

\end{document}